\documentclass[12pt,reqno,twoside]{amsart}
\usepackage{amsmath}
\usepackage{amssymb}
\usepackage[all]{xy}
\usepackage{epsfig}
\usepackage{color}
\title{Characterizations of lattice surfaces}
\author{John Smillie
}
\address{Cornell University, Ithaca, NY {\tt smillie@math.cornell.edu}}
\author{Barak Weiss}
\address{Ben Gurion University, Be'er Sheva, Israel 84105
{\tt barakw@math.bgu.ac.il}}
\font\sb = cmbx8 scaled \magstep0
\font\sn = cmssi8 scaled \magstep0

\long\def\combarak#1{\ifdraft{\sb #1 }\else\ignorespaces\fi}

\newcommand\hol{\mathrm{hol}}

\newif\ifdraft\drafttrue
\draftfalse
\newcommand\name[1]{\label{#1}{\ifdraft{\sn [#1]}\else\ignorespaces\fi}}

\newcommand\eq[2]{{\ifdraft{\ \tt [#1]}\else\ignorespaces\fi}\begin{equation}\label{#1}{#2}\end{equation}}
\newcommand {\equ}[1]{\eqref{#1}}

\newcommand{\HH}{{\mathcal{H}}}
\newcommand{\HHH}{{\mathbb{H}}}

\newcommand{\Q}{{\mathbb {Q}}}

\newcommand{\R}{{\mathbb{R}}}
\newcommand{\TT}{{\mathcal{T}}}
\newcommand{\ii}{{\mathbf{i}}}
\newcommand{\Z}{{\mathbb{Z}}}

\newcommand{\SL}{\operatorname{SL}}
\newcommand{\PSL}{\operatorname{PSL}}

\newcommand{\SO}{\operatorname{SO}}

\newcommand{\PSO}{\operatorname{PSO}}

\newcommand{\diag}{{\rm diag}}

\newcommand {\ignore}[1]  {}

\newcommand{\LL}{{\mathcal L}}

\newcommand{\FF}{{\mathcal{F}}}

\newcommand{\til}{\widetilde}

\newcommand{\sm}{\smallsetminus}

\newcommand{\vre}{\varepsilon}

\newcommand{\NST}{{\mathrm{NST}}}
\newcommand{\NSVT}{{\mathrm{NSVT}}}

\newcommand{\LCM}{{\mathrm{LCM}}}

\newcommand{\Aff}{{\mathrm{Aff}}}

\newtheorem{thm}{Theorem}[section]

\newtheorem{lem}[thm]{Lemma}

\newtheorem{prop}[thm]{Proposition}

\newtheorem{cor}[thm]{Corollary}

\newtheorem{remark}[thm]{Remark}

\newtheorem{example}[thm]{Example}

\newtheorem{question}{Question}

\begin{document}
\maketitle

\begin{abstract}
We answer a question of Vorobets by showing that the lattice property
for flat surfaces is equivalent to the existence of a positive lower
bound for the areas of affine triangles.
We show that the set of affine equivalence classes of lattice surfaces
with a fixed positive lower bound for the areas of triangles is finite
and we obtain explicit bounds on its cardinality. We deduce 
several other characterizations of the lattice property. 

\end{abstract}

\section{Introduction}

Our objects of study are translation and half-translation surfaces and
their affine automorphism groups. These structures arise in the study 
of rational polygonal billiards. They also arise in Thurston's classification of
surface diffeomorphisms in connection with measured foliations.
 Isomorphic structures arise in
complex analysis where they are called respectively abelian differentials
and quadratic differentials. We will use the term flat surface for both
translation and half translation surfaces when it is not important to distinguish
between the two types. For more details on flat surfaces see
\cite{Vorobets, MT, zorich survey}. 

Let $\Aff(M)$ denote the affine automorphism group of a flat surface $M$. 
For a typical flat surface this group is trivial; however surfaces
with non-trivial automorphism groups are quite interesting. 
 Taking the differential of the automorphism yields a
homomorphism $D: \Aff(M) \to G$ with finite kernel, 
where $G$ is either $\SL(2, \R)$ or $\PSL(2,\R)$ depending on whether
$M$ is a translation or half-translation surface. The image 
$\Gamma_M$ of this homomorphism is called the {\em Veech group} of $M$.  
We say that $M$ is a {\em lattice surface} if $\Gamma_M$ is a lattice,
i.e. has finite
covolume in $G$.
In a celebrated
paper \cite{Veech - alternative}, Veech
constructed a family of lattice surfaces, and 
showed that lattice surfaces have striking
dynamical properties, in particular they satisfy the `Veech dichotomy'
which will be discussed below.
Other examples have been constructed by several authors 
and classifications
of lattice surfaces are known in special cases (see \cite{GJ, KS, Puchta,
Calta, McMullen}),
but an overall classification does not yet exist.

Vorobets \cite{Vorobets} found a connection between the lattice
property and the collection of areas of triangles in the
surface. A flat surface is equipped with a finite set of distinguished
points $\Sigma = \Sigma_M$ which contains all cone points (those with
cone angle not equal to $2\pi$) and may contain other points. We
always assume $\Sigma \neq \varnothing$, and that elements of
$\Aff(M)$ preserve $\Sigma$. A {\em triangle in $M$} is the image
of an affine map from a triangle in the plane to $M$, which takes the
vertices of the triangle to points in $\Sigma,$ is injective on
the interior of the triangle, and such that interior points do not map
to $\Sigma$. In particular the vertices of the triangle need not be
distinct. The collection of areas of triangles will be denoted by
$\TT(M)$. Vorobets proved that $M$ is a lattice surface if and only if
$\TT(M)$ is finite. Say that $M$ {\em has no small triangles} if $\inf
\, \TT(M)>0$. Vorobets also showed that if $M$ has no small triangles
then it satisfies the Veech dichotomy, and raised the question of
whether the no small triangles property is equivalent to the lattice
property.

\begin{thm}
\name{thm: characterization}
A flat surface has the lattice property if and only if it has no small
triangles. 

\end{thm}

Let us consider only surfaces $M$ which have total area 1. For
$\alpha>0$, let 
$$\NST(\alpha) = \left\{M: \inf \, \TT(M) \geq \alpha \right \}.$$
Say that $M$ and $M'$ are {\em affinely equivalent} if there
is an affine homeomorphism between them; note that for surfaces of
area 1, $\mathcal{T}(M)$
depends only on the affine equivalence class of $M$. 

\begin{thm}
\name{thm: NST finite}
For any $\alpha>0$, 
$\NST(\alpha)$ contains a finite number of affine equivalence classes.
\end{thm}

See  Proposition
\ref{prop: explicit} for an explicit bound on this number, and see
Proposition \ref{prop: coarea} for a bound on the sum, over all $M \in
\NST(\alpha),$ of the co-areas of $\Gamma_M$.

Besides Theorem \ref{thm: characterization}, our results yield several
different characterizations of lattice 
surfaces, which we collect in Theorem \ref{thm: TFAE} below. To state
them we introduce some terminology.
A {\em saddle
connection} on $M$ is a straight segment on $M$ connecting points of
$\Sigma$ (which need not be distinct), with no points of $\Sigma$ in
its interior. Let $\LL= \LL_M$ denote the set of all
saddle connections on $M$, and let $\LL(\theta) = \LL_M(\theta)$ denote the
set of saddle connections in direction $\theta$.
We say that $\theta$
is a {\em periodic direction} if
each connected 
component of $M \sm \LL(\theta)$ is a cylinder with a waist curve in
direction $\theta$.
Veech \cite{Veech - alternative} showed that a lattice surface
satisfies the following dichotomy:
\begin{itemize}
\item[I.] If $\LL(\theta) \neq \varnothing$ then $\theta$ is
a periodic direction. 
\item[II.] If $\LL(\theta) = \varnothing$ then $\FF_\theta$ is 
uniquely ergodic. 
\end{itemize}
This dichotomy does not characterize the lattice property (see
\cite{SW -- Veech}), but a modification of property I introduced by
Vorobets does. 
Namely, say
that $M$ is {\em uniformly completely periodic} if there is $s>0$ such
that 
each $\theta$ for which $\LL(\theta) \neq \varnothing$ is a periodic
direction, and the ratio of 
lengths of any two segments in $\LL(\theta)$ does not exceed
$s$. We will show that this property characterizes the
lattice property. Moreover, for a lattice surface in $\NST(\alpha)$ we
will provide an effective estimate for $s$ in terms of $\alpha$, see
Proposition \ref{prop: s bound}. 

A periodic direction
$\theta$ is called {\em parabolic} if the moduli of
all the cylinders are commensurable, and $M$ is called {\em uniformly
completely 
parabolic} if it is uniformly completely periodic, with all periodic
directions parabolic. 

Another characterization involves the set of holonomy vectors of
saddle connections. Associated with $\delta \in \LL$ is a vector
hol($\delta$) in the plane 
of the same length and direction as $\delta$. We set 
$$\mathrm{hol}(M) = \{ \mathrm{hol}(\delta) : \delta \in \LL \}.
$$
This is a discrete $\Gamma_M$-invariant subset of the plane which has attracted
considerable attention. 
For $\beta>0$, we let 
$$\NSVT(\beta) = \left\{ M : \inf \{ |v_1 \wedge v_2| : v_i \in \mathrm{hol}(M), v_1 \wedge v_2
\neq 0 \} \geq \beta \right \},
$$
and say that $M$ has {\em no small virtual triangles} if it belongs
to some $\NSVT(\beta)$. 
\ignore{The {\em virtual triangle
spectrum} of $M$ is 
$$
\mathcal{VT} (M) =\{ |v_1 \wedge v_2| : v_i \in \mathrm{hol}(M) \}.
$$
We let 
$$
\NSVT(\alpha) = \{M : \inf \, \mathcal{VT}(M) \geq \alpha \}
$$
(NSVT stands for `no small virtual triangles'). }
We will show that having no
small virtual 
triangles is also equivalent to the lattice property. 

Two other characterizations involve the dynamics of the $G$-action on
the space of flat surfaces. 
Associated with $M$ is the topological data consisting of the
underlying surface, the number of points in $\Sigma$ and the
associated cone angles, and whether or not the corresponding
foliations are orientable. The set of all $M$ sharing this data is
a noncompact orbifold called a {\em stratum}. The restriction of the
action of $G$ to the subgroup 
$\{g_t\}$, where 
$$
g_t = \left(\begin{array}{cc}
e^{t/2} & 0 \\ 0 & e^{-t/2} \end{array} \right)
$$ 
is called the {\em geodesic flow}. 


\begin{thm}
\name{thm: TFAE}
The following are equivalent for a flat surface $M$:
\begin{itemize}
\item[(i)]
$M$ is a lattice surface.
\item[(ii)]
$M$ is uniformly completely periodic. 
\item[(iii)]
$M$ is uniformly completely parabolic.
\item[(iv)]
$\left|\mathcal{T}(M) \right| < \infty$.
\item[(v)]
The set of triangles for $M$ consists of finitely many $\Aff(M)$-orbits.
\item[(vi)]
$M$ has no small triangles. 
\item[(vii)]
$\{u \wedge v : u, v \in \hol(M)\}$ is a discrete set of numbers. 
\item[(viii)]
For any $T>0$, the set 
$$\{(\xi,\eta) \in \LL_M \times \LL_M :
|\hol(\xi) \wedge
\hol(\eta)| < T\}$$
contains finitely many $\Aff(M)$-orbits. 
\item[(ix)]
$M$ has no small virtual triangles. 
\item[(x)]
The $G$-orbit of $M$ is closed.
\item[(xi)]
There is a compact subset $K$ of the stratum containing $M$ such
that for any $g \in G$, the geodesic orbit of $gM$ intersects $K$. 
\end{itemize}
\end{thm}

Many of these implications are due to Vorobets \cite[\S
6]{Vorobets}. The main new implication is (vi) $\implies$ (i). Our
first 
proof of this implication is sketched in \cite{toronto}. 

The proofs of Theorems \ref{thm: characterization} and
\ref{thm: NST finite} are similar. Vorobets showed that if $M$
has no small triangles any saddle connection direction is parabolic,
so associated with two non-parallel saddle connections are
a pair of cylinder decompositions of $M$, each with cylinders of
commensurable moduli. 
Moreover a bound on triangle areas leads to an upper bound on certain 
combinatorial data associated with two cylinder
decompositions. Arguments of Thurston and Veech show that such 
combinatorial data determine a flat surface up to affine
equivalence. This immediately 
yields Theorem \ref{thm: NST finite}. 

To derive Theorem \ref{thm:
characterization} we introduce the {\em spine} $\Pi$ of a translation
surface $M$; 
this is a $\Gamma_M$-invariant tree in $\HHH$, whose edges are
labeled by two saddle connections on $M$, and correspond to the surfaces
affinely equivalent to $M$ for which these saddle
connections are simultaneously shortest. There is a retraction $\rho:
\HHH \to \Pi$ taking a surface in which all shortest saddle
connections are parallel, to an affinely equivalent surface which has
shortest saddle connections in two or more directions. For each edge $e$ of
$\Pi$, $\rho^{-1}(e)$ is a finite-area domain in $\HHH$, and $e$
is associated with a pair of cylinder decompositions. A bound on
triangle areas bounds the combinatorial data and thus bounds the
number of  edges in $\Pi/\Gamma_M$, giving a
bound on the area of $\HHH/\Gamma_M$. 

Theorems \ref{thm: characterization} and \ref{thm: NST finite} suggest a
natural ordering on the set of lattice surfaces. For a lattice
surface $M$, let $\alpha(M)$ denote the largest $\alpha$ for which $M
\in \NST(\alpha)$. 
Then all (affine equivalence classes of) lattice surfaces may be
written as 
\eq{eq: alphaM}{
M_1, M_2, \ldots, \ \ \mathrm{with \ \ }\alpha(M_1) \geq \alpha(M_2)
\geq \cdots,
}
and a similar ordering exists with the NSVT condition instead of the NST
condition. 

It would be desirable to have an algorithm which,
given $\alpha >0$, lists all surfaces in $\NST(\alpha)$ or
$\NSVT(\alpha)$. 
Our analysis yields an explicit finite set of
surfaces, presented in terms of Thurston-Veech combinatorial
data, which contains all the above lattice surfaces, and implicitly, an
algorithm for distinguishing the lattice from the non-lattice
surfaces. This is our motivation for providing effective proofs 
and explicit estimates for the cardinality of certain finite sets. 
We suspect our explicit estimates are far from optimal and
believe additional theoretical work would be required in order to
create an algorithm which is practically feasible.

In the interest of presenting the
simplest formulae, our estimates may be presented in terms of either
triangle or virtual triangle areas. We do not present 
estimates involving the genus of $M$ or the stratum
containing $M$, both because these are harder to obtain, and because
the same effective search will produce the lattice surfaces in all
genera. 

 The relation between
the no small triangles and no small virtual triangles properties is
described in the following:
\begin{thm}
\name{thm: nst vs nsvt}
$\NSVT(\alpha/2) \subset \NST(\alpha) \subset \NSVT(2\alpha
e^{-1/(2\alpha e)})$. 

\end{thm}

Without conducting computer searches, we are able to
determine $M_1, \ldots, M_7.$ They are all
arithmetic, and comprise $\NST\left(\frac16\right).$

\medskip
{\bf Acknowledgements.} We thank Alex Eskin for drawing our attention
to \cite{Vorobets} at an early stage of this work, and we thank Yaroslav Vorobets for 
useful discussions. The authors gratefully acknowledge support from NSF grant
DMS-0302357, BSF grant 2004149 and ISF grant 584/04.

\section{Basics}
%
We review some definitions here. For more details we refer the reader to 
\cite{MT, Vorobets, zorich survey}.
Throughout this paper, $M$ denotes a connected oriented surface with a flat
structure. When there is no danger of confusion, we will also denote
by $M$ the underlying topological surface. We denote the genus of $M$
by $g$. 
A half-translation structure may be
thought of as an equivalence class of atlases 
of charts $(U_{\alpha}, \varphi_{\alpha})$ covering all but a
non-empty finite
set $\Sigma = \Sigma_M$, such that
the transition functions $\R^2 \to \R^2$ are of the form $\vec{x}
\mapsto  \pm \vec{x} + \vec{c}$, and such that around each $\sigma \in
\Sigma$ the charts glue together to form a cone point with cone angle
$2\pi(r+1)$, where $r =r_{\sigma}\in \{ -\frac12,   \frac12, 1,
\frac32, 
\ldots\}$. 
A translation
structure is similarly defined, with the  
requirement that the transition functions are of the form $\vec{x}
\mapsto \vec{x} + \vec{c}$. 
Points in $\Sigma$ are called {\em singularities}. 
Several authors permit singularities $\sigma$ 
for which $r_\sigma =0$; such points are called {\em removable
singularities} or {\em marked points}. Except in \S\ref{section:
simplest}, we will always 
assume that our singularities are not removable.

An orientation-preserving homeomorphism $\varphi: M_1 \to M_2$ which is
affine in each chart is 
called an {\em affine isomorphism}, or an {\em affine automorphism} if
$M_1=M_2$. We denote the set of affine automorphisms of $M$ by
$\Aff(M)$. An affine isomorphism whose linear part is $\pm 
\mathrm{Id}$ is called an {\em translation equivalence}. 
The map $D :\Aff(M) \to G$ which assigns to
$\varphi$ its linear part has a finite kernel, consisting of 
translation equivalences of $M$. 
Two atlases are considered equivalent if there is a translation
equivalence betweeen them.

There is a standard {\em orientation
double cover} construction which associates to each half-translation
surface $M$ a translation surface $M'$ with a branched degree two
translation cover $M' \to M$; thus many statements about flat surfaces
can be reduced to statements about translation surfaces. 

 A standard analogue of the Gauss-Bonnet formula (see \cite{Vorobets})
says that 
\eq{eq: Gauss Bonnet}{\sum r_{\sigma} = 2g-2.}
Summing the total angle at all the singularities we obtain the number 
\eq{eq: defn tau}{
\tau=\tau(M) = \sum_\sigma 2 \pi(r_\sigma+1) = 2\pi\left(2g-2+ |\Sigma|
\right). 
}
Since each triangle in a triangulation of $M$ with vertices in
$\Sigma$ contributes a total 
angle of $\pi$, the number of triangles in such a triangulation is
exactly $\tau/\pi$.

A flat surface inherits a
Euclidean area form from 
the plane and we normalize our surfaces by assuming that each surface has unit area.
An affine automorphism of $M$ is a  
self-homeomorphism which is affine in each chart.

For a flat surface $M$ let $\vec{r}_M = \left(r_{\sigma} \right)_{\sigma \in
\Sigma_M}$. The set of all translation equivalence classes of 
(half-) translation surfaces of a given combinatorial type, namely those for
which the data $\vec{r}_M$ is fixed, 
is called a {\em stratum}. 
Each stratum is equipped with a structure of
an affine orbifold, which is locally modeled on $H^1(M, \Sigma; \R^2)$
(in the case of translation surfaces), or an appropriate subspace of
the $H^1(\til M, \til \Sigma; \R^2)$ (in the case of half-translation
surfaces, where $\til M, \til \Sigma$ are the orientation double cover
of $M, \Sigma$). 

There is an action of $G$ on $\HH$ by post-composition on each
chart in an atlas. It follows from the above that $\Gamma_M = \{g \in
G: gM=M\}.$ 
We define 
\[
r_{\theta}=
\left(\begin{array}{cc}
\cos \theta & -\sin \theta \\
\sin \theta & \cos \theta
\end{array}
\right),\ \ \ \ \ 
h_s 
= \left(\begin{array}{cc} 1 & s \\ 0 & 1 
\end{array}
\right),
 \ \ \ \ \, \ \til h_s = \left(\begin{array}{cc} 1 & 0 \\ s & 1 
\end{array}
\right),
\]
which we view as one-parameter subgroups of $G$.

By a {\em direction} we mean an element of $P(\R^2)$ (resp. $S^1$) in
the case of half-translation (resp. translation) surfaces. There is a
natural action of $G$ on the set of directions. 

A {\em cylinder} for $M$ is a topological annulus on the surface which is
isometric to $\R /c \Z \times (0,h)$, where $c$ is the {\em
circumference} of the cylinder and $h$ is its {\em height}.
We say that the cylinder is {\em horizontal (vertical)} if the angle
between the horizontal 
direction and the direction of the waist curve is $0$
(resp. $\pi/2$). Note that if the cylinder is vertical then the height
actually measures its horizontal width. 
A cylinder is {\em maximal} if it is not contained in a
larger cylinder, and this implies that both of its boundary components
 contain singularities. 
We define $\LL(\theta)$ to be the collection of closed saddle connections
in direction $\theta$. Recall that if $\theta$ is a periodic direction
then $M$ has a {\em cylinder decomposition in direction
$\theta$} i.e., the complement of $\LL(\theta)$ is a union of
cylinders. 
An automorphism $\varphi \in \Aff(M)$ is called {\em parabolic} if
$D\varphi \in \Gamma_M$ is a parabolic matrix. 
The 
{\em inverse modulus} of a cylinder as above is $\mu=w/h$. 
We say that real numbers $\mu_1, \ldots, \mu_k \in \R$ are {\em commensurable} if
$\mu_i/\mu_j \in \Q$ for all $i,j \in \{1, \ldots, k\}$. If this
holds, we denote by $\mu=\LCM(\mu_1, \ldots, \mu_k)$ the smallest
positive number which is an integer multiple of all the $\mu_i$, and
we call $(n_1, \ldots, n_k)$, where $n_i = \mu /\mu_i$, the
{\em Dehn twist vector} corresponding to the 
cylinder decomposition. Note that by definition $\gcd(n_1, \ldots,
n_k)=1$. This terminology is motivated by the fact if $\varphi \in
\Aff(M)$ is such that $D\varphi$ is parabolic and fixes the direction
$\theta$, then there is a cylinder
decomposition in direction $\theta$ with 
the inverse moduli of the cylinders commensurable.
Conversely, given a decomposition of $M$ into cylinders for which
the $\mu_i$ are commensurable, one can construct an associated parabolic 
affine automorphism. We say that $\varphi$ is {\em simple}
if $\varphi$ fixes the saddle connections $\LL(\theta)$
pointwise. This implies that each cylinder is taken to itself and that
the map on each cylinder is topologically a certain number of Dehn
twists.  Let $n_1, \ldots, n_k$ be a collection of natural numbers and
let $\mu \neq 0$ satisfy the equation $n_j\mu_j=\mu$. Then there is a simple parabolic automorphism
$\varphi$ with derivative $D\varphi=r_\theta h_\mu r_{-\theta}$ which induces $n_j$ Dehn
twists in the cylinder $C_j$. 

\section{Cylinder decompositions}
Vorobets showed that cylinder decompositions arise in connection with
the no small triangles condition: 

\begin{prop}[Vorobets]
\name{prop: Vorobets parabolic}
 If $M$ has no small triangles then any
direction $\theta$ for which $\LL(\theta) \neq \varnothing$ is
parabolic.
\end{prop}

In addition Vorobets obtained bounds on the ratios of lengths of saddle connections in 
$\LL(\theta)$ and the Dehn twist numbers $n_j$. We will review
Vorobets' arguments here and use them to relate the no small triangles
and the no small virtual triangle conditions. Where possible we will
improve the constants that arise.

\begin{prop}
\name{prop: s bound}
If $M \in \NST(\alpha)$ then $M$ is $s$-uniformly periodic for $s =
\min \left\{e^{1/(2\alpha e)},  (2(r-1)\alpha)^{1-r}\right\}$, where
$r$ is the maximal number of cylinders in a cylinder decomposition on
$M$.  
\end{prop}
\begin{proof}
Suppose $\theta$ is a saddle connection direction for $M \in
\NST(\alpha)$. We know from Proposition \ref{prop: Vorobets parabolic} that
$\theta$ is a periodic direction on $M$. Let $\sigma, \sigma' \in
\LL_M(\theta)$. Suppose first that $\sigma, \sigma'$ are on the
boundary of the same cylinder $C$, of area $A$. Let $h$ be the height
of $C$, and let $x, x'$ be the lengths of $\sigma,
\sigma'$. Then $x h \leq A$ since the circumference of $C$ is at
least $x$, and $x' h \geq 2\alpha$ since $C$ contains a triangle of
area $x'h/2$. This implies that
\eq{eq:3}{
\frac{x}{x'} = \frac{xh}{x'h} \leq \frac{A}{2 \alpha}.
}
Now suppose $\sigma, \sigma'$ are not on the boundary of the same
cylinder. Since $M$ is connected, there is a chain of cylinders 
$C_1,\ldots C_{j}$ such that $\sigma = \sigma_1$ is on the boundary of
$C_1$ and $\sigma'=\sigma_{j+1}$ is on the boundary of $C_j$, and $C_i$ and $C_{i+1}$
meet along a segment $\sigma_i$ of length $x_i$ for $i=1, \ldots,
j$. Applying \equ{eq:3} $j$ times we have  
\begin{equation}
\name{eq:product}
\frac{x}{x'}=\prod_{i=1}^{j}\frac{x_i}{x_{i+1}}\le\prod_{i=1}^{j}\frac{A_i}{2
\alpha}.
\end{equation}
Since $M$ has area 1 and the cylinders $C_i$ are disjoint,
\begin{equation}
\name{eq:area}
\sum A_i\le 1
\end{equation}
The maximum of the right hand side of \equ{eq:product} subject to the
constraint \equ{eq:area}  
occurs when when $A_i=1/j$, so that 
$$\frac{x}{x'}\le (2 j \alpha)^{-j}. $$

The maximum of  $(2 j\alpha)^{-j}$ as a function of $j$ occurs when
$j=1/(2 \alpha e)$.
If $1/(2\alpha e)\ge r-1$ then the maximum value of ${2j\alpha}^{-j}$ occurs when
$j=r-1$. In this case the maximum value of ${2 j\alpha}^{-j}$ is $(2(r-1)\alpha)^{1-r}$.
If $1/(2\alpha e)\le r-1$ then the maximum value of $(2j\alpha)^{-j}$ occurs when $j=r-1$.
In this case the maximum value of $(2j\alpha)^{-j}$ is $e^{1/(2\alpha e)}$.
\end{proof}

We deduce the following statement, which implies Theorem 
\ref{thm: nst vs nsvt}. 
\begin{prop}
\name{prop: nst vs nsvt}
If $M \in \NST(\alpha)$ and the number of cylinders in a decomposition of
$M$ is at most $r$, then $M \in \NSVT(\beta)$ where 
$$\beta = 2\alpha \max \left\{e^{-1/(2\alpha e)}, \left(2(r-1)\alpha
\right)^{r-1} \right\}.$$ 

\end{prop}
\begin{proof}Let $M\in
\NST(\alpha)$. Let $\xi$ and $\eta$ be nonparallel saddle
connections on $M$. By applying an element of $G$ we may assume that $\xi$
is horizontal and $\eta$ vertical. Since $M \in \NST(\alpha)$ there is a
decomposition of $M$ into horizontal cylinders, so there is a saddle
connection $\xi'$ parallel to $\xi$ making an angle of $\pi/2$ with
$\eta$ at one of the endpoints of $\eta$. By Proposition \ref{prop: s
bound} we have $x' \leq sx,$ where $x, x'$ are the lengths of $\xi,
\xi'$ and $s = \min \left\{e^{1/(2\alpha e)},  (2(r-1)\alpha)^{1-r}\right\}$.

Consider the right triangle $\Delta$ in the plane with horizontal and
vertical sides the same length as $\xi', \eta$. If there is an
isometric copy of $\Delta$ embedded in $M$ under an affine mapping
sending its edges to $\xi', \eta$, we would have $|\xi' \wedge \eta|
/2  = \mathrm{area} \, \Delta \geq \alpha$. If there is no such isometric
copy of $\Delta$ then $M$ contains an isometric copy of a triangle
contained in $\Delta$, and again $|\xi' \wedge \eta| \geq 2 \alpha$.
Therefore 
$$\left|\xi \wedge \eta \right| \geq \frac1s |\xi' \wedge \eta| \geq
\frac{2\alpha}{s} = \beta.
$$
\end{proof}

\begin{remark}
When applying Propositions 
\ref{prop: s bound}, \ref{prop: nst vs
nsvt} it is useful to 
have a bound for the maximal number of cylinders in a cylinder
decomposition on $M$. 
A standard bound of $2g + |\Sigma| -2 $ can be found in e.g. \cite[proof of Lemma,
p. 302]{KMS}. A better bound is 
$g+i+\left \lfloor \frac{j}{2} \right \rfloor -1,$
where $g$ is the genus of $M$, and $i$ (resp. $j$) is the number of
even (resp. odd) -angled singularities of $M$. A proof of this bound,
due to the first-named author, is given 
in \cite{Naveh} for the case of translation surfaces; the same proof
works for half-translation surfaces.
\end{remark}

\ignore{

An automorphism $\varphi \in \Aff(M)$ is called {\em parabolic} if
$D\varphi \in \Gamma_M$ is a parabolic matrix. 
One way in which cylinder decompositions arise is in connection with parabolic automorphisms. 

\begin{prop}[Vorobets] If $M$ has no small triangles then the cylinder
decomposition that arises in each saddle connection direction is
parabolic. 
\end{prop}
}
For the cylinder decompositions that arise via Proposition \ref{prop: Vorobets
parabolic}, Vorobets
bounded the number and size of the Dehn twist numbers in terms of the
lower bound for areas of triangles. We give our version here.

\begin{prop}
\name{prop:diagonal count}
The number of  Dehn twist vectors for a 
cylinder decomposition of a surface in $\NST(\alpha)$ into  $r$ cylinders
is bounded above by:
$$(8\alpha)^{2(1-r)}(1-2\log 8 \alpha)^{r-1}
.$$
\end{prop}

\begin{proof}
Let $M=C_1 \cup \cdots \cup C_r$ be a cylinder decomposition, where the cylinders have
inverse moduli $\mu_1,\ldots, \mu_r$. We can recover the Dehn twist
vector $\vec{n}$ from the numbers $\mu_i/\mu_j$.

We first recall Vorobets' argument which gives a restriction on
the ratio of moduli of neighboring cylinders. Say that we have a pair of cylinders
$C_i$ and $C_j$ with an edge in common. Let $h$ and $c$ be the height and
circumference of a cylinder and let $A$ be its area. 
As shown in \cite[proof of Prop. 6.2]{Vorobets}, applying the $G$-action and
Dehn twists in the cylinders one can find a triangle 
in $C_i \cup C_j$ with base $h_i$ and height in $\displaystyle{w_0 + \Z \mu_i h_j +
\Z w_j = w_j\left(\frac{\mu_i}{\mu_j} \Z + \Z \right)}$ for some $w_0$. 
Since for $\mu_i/\mu_j  = p/q$ we have that $\displaystyle{\Z +
\frac{\mu_i}{\mu_j}\Z}$ is $1/2q$-dense in $\R$, we have by our hypothesis
\begin{equation*}
q\le\frac{h_iw_j}{4\alpha}.
\end{equation*}
If we interchange the roles of the cylinders we get 
$
p\le h_jw_i / (4\alpha),
$
and multiplying the equations gives
\eq{eq: *}{
pq\le \frac{h_iw_ih_jw_j}{16\alpha^2}=\frac{A_iA_j}{16\alpha^2}.
}
The maximum of the right hand side of \equ{eq: *} under the constraint
$A_i+A_j \leq 1$ is attained when $A_i = A_j =1/2$ so we get 
$$pq \leq \frac1{64\alpha^2}.$$

To count the possible values we note that the number of pairs of natural numbers $(x,y)$ that
satisfy $xy\le K$ is bounded by the area under the graph of $y=\min(K,\log x)$ from $x=0$ to $x=K$
and this area is $K(1+\log K)$.
Thus the number of possible ratios of moduli of neighboring cylinders is bounded by
$\displaystyle{
\frac{1}{64\alpha^2}\left(1+\log \frac{1}{64\alpha^2}\right).
}$

The connections between cylinders can be recorded in a graph whose
vertices are cylinders and whose edges correspond to cylinders
which share a segment. Since $M$ is connected this graph is connected. 
In order to recover $\vec{n}$ it suffices to
know the ratios of moduli for a collection of edges that span a maximal tree
in this graph. The number of edges in a maximal tree is $r-1$, and we
obtain our bound of 
\[
\prod_{edges}\frac{1}{64\alpha^2}\left(1+\log
\frac{1}{64\alpha^2}\right )= (8\alpha)^{2(1-r)}(1-2\log 8\alpha)^{r-1}.
\]
\end{proof}

We will also need a similar bound in terms of $\beta,$ when $M \in
\NSVT(\beta)$. Note that this bound is independent of $r$. 
\begin{prop}
\name{prop: twist bounds beta}
If $M \in \NSVT(\beta)$ and $\vec{n} = (n_1, \ldots, n_r)$ is a Dehn
twist vector for a parabolic direction on $M$ with a decomposition
into $r$ cylinders, 
then for any $i, j$:
\eq{eq: twist bounds beta}{
pq \leq \frac{1}{64\beta^2}, \ \ \ \ \  \mathrm{where} \ \ \frac{n_i}{n_j} =
\frac{p}{q} \ \ \mathrm{and} \, \ \gcd(p,q)=1.
}
\end{prop}

\begin{proof}
By renumbering assume $i=1,\ j=2$. Let $C_1, C_2$ be the corresponding
cylinders and let $\delta_1, \delta_2$ be saddle connections passing
across the cylinders, so that $\hol(\delta_i) = (x_i, h_i)$, where
$|h_i|$ is the height of $C_i$ and $|x_i|$ is no greater than the
circumference $c_i$ of $C_i$. By
applying Dehn twists in $C_1, C_2$ and the $G$-action, we find an
affinely equivalent surface for which $\hol(\delta_1) = (0, h_1)$ and
$\hol(\delta_2) = (x, h_2)$ with $|x| \leq c_1/2q$. This implies
$$\beta \leq |\hol(\delta_1) \wedge \hol(\delta_2)| = |xh_2| \leq 
\frac{c_1h_2}{2q}.$$
Interchanging the roles of $C_1, C_2$ we obtain $\beta \leq
\frac{c_2h_1}{p}$ so that 
\eq{eq: Lagrange}{
pq \leq \frac{A_1A_2}{4\beta^2},
}
where $A_i$ is the area of $C_i$. 
 The
maximum of the right hand side of \equ{eq: Lagrange} subject to the
requirement $A_1+A_2 \leq 1$ is obtained when $A_1 = A_2 = 1/2$, and
we obtain 
\equ{eq: twist bounds beta}.
\end{proof}

\section{The spine and a criterion for finite covolume}
Given a flat surface $M$ of area 1 we can consider the collection of all
area 1 surfaces $M'$ affinely equivalent to $M$. An affine equivalence
$f: M \to M'$ determines an element $g = Df \in G$, whose coset
$g\Gamma_M$ in $G/\Gamma_M$ depends only on $M'$. This establishes an
identification of $G/\Gamma_M$ with its affine equivalence class.  
The Veech group $\Gamma = \Gamma_M$ is the isotropy
group of $M$ so we can identify this family with
$G/\Gamma$. Identifying two flat surfaces which are isometric, we find
that the collection of isometry classes of surfaces affinely equivalent to
$M$ can be identified with $\PSO(2, \R)\backslash G/\Gamma$ or
$\HHH/\Gamma$ where $\HHH$ is the hyperbolic plane. The image of this
quotient in moduli space is also called the Teichm\"uller 
disk associated to $M$. 

We define $\lambda(M)$ to be the infimum of lengths of saddle
connections in $M$. Since the collection of lengths of saddle
connections is discrete this infimum is achieved and is positive. The
function $\lambda$ depends only on the isometry class of $M$ so it is
well defined on $\HHH/\Gamma$. It is also useful to view it as a
$\Gamma$-invariant function on $\HHH$. 

The function $\lambda$ is bounded above.

\begin{prop} \name{prop: go bears}
If $\tau = \tau(M)$ is as in \equ{eq: defn tau} then 
$\lambda(M)\le\sqrt{2/\tau}.$
\end{prop}

\begin{proof} Write $\lambda$ for $\lambda(M)$. Let $M_0\subset M$
consist of points $p\in M$ for which the distance from $p$ to a point
of $\Sigma$ is less than $\lambda/2$, so that $p$ is closest to a
unique $q\in \Sigma$. If we fix a $q\in\Sigma$  
then the set of points in $M_0$ closest to $q$ has area $c_q\lambda^2/2$ where
$c_q$ is the cone angle at $q$. By \equ{eq: defn tau}, 
the area of $M_0$ is $\tau\lambda^2/2$. Since the total area 
of $M$ is 1 we have $\lambda\le\sqrt{2/\tau}$.
\end{proof}

We say a saddle connection $\sigma$ {\em has minimal length} if
$\ell(\sigma)=\lambda(M)$. The number of minimal length saddle
connections in $M$ is finite and is generically one. 
Let $\Pi$ be the subset of $\HHH$ corresponding to surfaces affinely
isomorphic to $M$ which
have at least two non-parallel minimal length saddle connections. This
is a locally finite geodesic subcomplex of
$\HHH$ which we call the {\em spine} of $M$. 
A related but distinct construction is the tesselation studied in Veech \cite{Veech
- spine} and Bowman \cite{Bowman}.  

The edges of $\Pi$ correspond to surfaces with minimal length saddle
connections in exactly two directions. These edges are
geodesic segments in $\HHH$. Indeed, if $\xi$ and $\eta$ are minimal length saddle
connections in distinct directions on $M$ then the vectors $v=\hol(\xi)$
and $w=\hol(\eta)$ have the same length. We can rotate the translation
structure so that the vector $v+w$ is horizontal and the vector $v-w$
is vertical. In these coordinates $v=(x,y)$ and $w=(x,-y)$ for some
$x, y \in \R$. If we apply the geodesic flow to $M$ the holonomy vectors
of these saddle connections with respect to $g_t(M)$ are
$(e^{t/2}x,e^{-t/2}y)$ and $(e^{t/2}x,-e^{-t/2}y)$, so the
vectors continue to have the same length. In fact this geodesic in
$\HHH$ is exactly the set of translation structures for which $\xi$
and $\eta$ have the same length. The edge in $\Pi$ that contains $M$
is a segment contained in this geodesic.

Clearly $\Pi$ is $\Gamma_M$-invariant. 
The vertices of $\Pi$ correspond to surfaces with minimal length
saddle connections in three or more directions. 
We will use $\Pi$ to get a condition for the lattice property.

\begin{prop}
\label{prop: area est} The area of $\HHH/\Gamma_M$ is bounded above by
$2\pi N$ where $N$ is the number of edges in $\Pi/\Gamma_M$. 
\end{prop}

\begin{proof}
Write $\Gamma = \Gamma_M$. We define a retraction $\rho: \HHH/\Gamma \to
\Pi/\Gamma$. We will describe $\rho$ as a $\Gamma$-invariant map
from $\HHH$ to $\Pi$. Let $z \in \HHH$. It corresponds to a set of
isometric translation surfaces, i.e. to $\SO(2, \R) \cdot M'$ for some
$M'$ in the $G$-orbit of $M$. If $M'$ has minimal length saddle
connections in two or more non-parallel directions then $M'$
represents a point in $\Pi$ 
and we define $\rho(z)=z$. If all minimal length saddle connections on
$M$ are in the same direction, by rotating the coordinate system on
the surface we can assume 
that this direction is horizontal. Now apply the geodesic flow $g_t$
to the rotated surface. The length of the shortest saddle connection
$\sigma$ in the 
surface $g_t(M')$ is $e^{t/2}$ times the original length of
$\sigma$. As long as $g_t(M')$ is in the complement of $\Pi$, $\sigma$
remains a minimal length segment. Since the length of the shortest
curve in $g_t(M')$ is bounded above, for some $t_0$ we have $g_{t_0}(M')
\in \Pi$, and we define $\rho(z)$ to be the isometry class of $g_{t_0}(M')$. 

We now show that the set of all points which retract to a given
point corresponding to the surface $M_0$ in $\Pi$ is a union of geodesic
rays, one for each direction of a minimal length saddle connection
in $M_0$. Pick a 
saddle connection $\sigma$ of minimal length in $M_0$.  
By rotating the coordinate system assume that $\sigma$ is
horizontal. Consider the collection of surfaces that we get by flowing
by $g_{-t}$ for $t>0$. The matrix for $g_{-t}$ contracts the
horizontal direction more than any other direction so in each of these
surfaces $g_{-t}M_0$ the saddle connection $\sigma$ still has minimal length
and is shorter than all saddle connections in other directions. It
follows that for each of these surfaces the retraction takes $g_{-t}M_0$
to $M_0$. 

Pick an edge $e$ of $\Pi$. There are minimal length saddle
connections in two directions. Pick one direction and let $\sigma$ be
a saddle connection in this direction.  
The collection of surfaces which retract to $e$ and 
correspond to the direction of $\sigma$ in $\Pi$ is a union of
geodesic rays all of which are asymptotic to the same point on the
boundary of $\HHH$ (which corresponds to the direction of $\sigma$). It
follows that the set of points that retract to $e$ consists of two
geodesic triangles each with one vertex at infinity. 
Since the area of a geodesic triangle is bounded above by $\pi$ the
total area of $\HHH/\Gamma$ is bounded above by $2\pi N$.  
\end{proof}

\begin{remark}
The existence of the retraction $\rho: \HHH \to \Pi$ implies that
$\Pi$ is homotopy-equivalent to $\HHH$, i.e. a tree. 
\end{remark}

As we explained each edge $e$ in $\Pi$ is a geodesic segment
corresponding to the surfaces for which two saddle 
connections, say $\xi$ and $\eta$ have the same length. There is a
unique 
point on the geodesic containing $e$ corresponding to a surface $M_0$
where the 
saddle connections $\xi$ and $\eta$ have equal length and are perpendicular. It is
conceivable that $M_0$ is not in $e$, but nevertheless we have: 

\begin{lem}\name{lem: conceivable}
In the surface $M_0$ the common length of $\xi$ and $\eta$ is no more than
$\sqrt{2/\tau},$ where $\tau$ is as in \equ{eq: defn tau}. 
\end{lem}

\begin{proof} Let $M'$ correspond to a point in
$e$. By Proposition \ref{prop: go bears}, in this surface the lengths
of $\xi$ and $\eta$ are bounded above by 
$\sqrt{2/\tau}$. Since the $g_t$-action preserves 
$|\xi \wedge \eta| = \ell(\xi) \ell(\eta) \sin \theta$, where $\theta$
is the angle between $\xi$ and $\eta$,  
the minimum of $\ell(\xi)$  is
attained when $\xi$ and $\eta$ are perpendicular. Thus the lengths at
$M_0$ are no greater than
the lengths at $M'$, and we can use Proposition \ref{prop: go bears}. 
\end{proof}

We say that a translation surface $M$ is in {\em standard form} if it has
vertical and horizontal saddle connections, and if the shortest
horizontal and vertical saddle connections have the same length and
this length is bounded by $\sqrt{2/\tau(M)}$. 

\begin{cor}\name{cor: standard}
The area of $\HHH/\Gamma_M$ is at most $2\pi E$ where $E$ is the
number of standard form surfaces affinely equivalent to $M$. 
\end{cor}
~\qed
\section{The complexity of pairs of cylinder decompositions}

A pair of cylinder decompositions in distinct directions determine a
decomposition of the surface into parallelograms. We measure the
complexity of the pair by counting the number of parallelograms in
this decomposition. We will show that a lower bound on the area of
virtual triangles allows us to find a pair of cylinder decompositions
for which we can bound the complexity.

\begin{prop}
\name{prop: rectangle number bound}
Suppose  $M \in \NSVT(\beta_0)$ and $\xi$ and $\eta$ are saddle
connections on $M$ with $|\hol(\xi)\wedge\hol(\eta)|=\beta \ge \beta_0$. 
Then the number of rectangles in the pair of cylinder
decompositions determined by $\xi$ and $\eta$ is no more than
$\beta/\beta_0^2$. 
\end{prop}

\begin{proof} By changing coordinates using an element of $G$ we may assume
that $\xi$ is horizontal and $\eta$ is vertical, and they have the
same length. The product of the
lengths is $\beta$ so both lengths are 
$\sqrt{\beta}$. Let $h$ be the height of a horizontal cylinder. There
are singularities on the top and bottom of this cylinder and a saddle
connection $\sigma$ between them. The quantity
$|\hol(\sigma)\wedge\hol(\xi)|$ is just $h\sqrt{\beta}$. By
assumption $|\hol(\sigma)\wedge\hol(\xi)|\ge \beta_0$ 
so $h\ge\beta_0/\sqrt{\beta}$. A similar
calculation for the width of the vertical cylinders gives
$w\ge\beta_0/\sqrt{\beta}$. It follows that the area of any rectangle
is at least $\beta_0^2/\beta$. Since the total area of $M$ is 1, the
number of cylinders is at most  $\beta/\beta_0^2$.  
\end{proof}

\begin{cor} 
\name{cor: stand form equiv class} A surface in $\NSVT(\beta)$ is
affinely equivalent to a
standard form surface with at most $1/\beta$ rectangles. 
\end{cor}

\begin{proof} 
This follows from Proposition \ref{prop: rectangle number bound} with
$\beta = \beta_0$.  
\end{proof}

\begin{cor} 
\name{cor: stand form area est}
A standard form surface in $\NSVT(\beta)$ has a decomposition into at
most  $2/(\tau \beta^2)$ rectangles. 
\end{cor}

\begin{proof} Let $\xi$ and $\eta$ be horizontal and vertical saddle
connections of common length at most $\sqrt{2/\tau}. $ Then 
$|\hol(\xi)\wedge\hol(\eta)| \le 2/\tau$. By Proposition \ref{prop:
rectangle number bound}, the 
number of rectangles in the corresponding pair of cylinders
decomposition is bounded above by $2/(\tau\beta^2)$. 
\end{proof}



\section{Pairs of cylinder decompositions}
\name{section: pairs of parabolics}
In this section we review ideas of Thurston \cite{Thurston} and Veech
\cite{Veech - alternative} that make it possible to bound the number
of standard form surfaces in $\NSVT(\beta).$ To this end assume $M$ is
a flat surface for which the horizontal and vertical directions
are periodic. 
The intersections of 
the horizontal and vertical cylinders divide the surface into
rectangles. Let us label the rectangles with numbers $1 , \ldots,
\ell$ and call the resulting surface a {\em labeled surface}. Let
$S_\ell$ denote the permutation group on $\ell$ 
symbols. Assume first that $M$ is a translation surface, so that each
rectangle has a well defined top, bottom, left and right 
edge. For $k \in \{1, \ldots, \ell\}$ let $\sigma_1(k)=k'$ if the
$k'$th rectangle is attached to the right side of the $k$th
rectangle. Let $\sigma_2(k)=k'$ if the $k'$th rectangle is attached to
the top side of the $k$th rectangle. As in \cite{EO1} we observe that
connectedness of $M$ is equivalent to the condition that group generated by
$\sigma_1$ and $\sigma_2$ acts transitively on $\{1, \ldots, \ell\}$.
The collection of horizontal cylinders corresponds to the set of
cycles of $\sigma_1$ and the collection of vertical cylinders
corresponds to the set of cycles of $\sigma_2$. The orders of the
singular points of $M$ is determined by the cycle structure of the
commutator of $\sigma_1$ and $\sigma_2$. 
If $M$ is a half-translation surface, a gluing pattern can be defined
in a similar way. For example this can be defined using the
orientation double cover. We will omit the details. 

We can also reverse this process. Say that $\sigma_1$ and
$\sigma_2$ are elements of $S_\ell$, for which 
the group
generated by $\sigma_1$ and $\sigma_2$ acts transitively on $\{1,
\ldots, \ell\}$. Let $\vec{a}^{(1)},
\vec{a}^{(2)}$ be vectors with positive real entries, indexed by the
cycles of $\sigma_1, \sigma_2$. 
Then there is a labeled translation surface $M$ where the heights of
the horizontal (vertical) cylinders are recorded in $\vec{a}^{(1)}$
(resp. $\vec{a}^{(2)}$); indeed, let $R_1, \ldots, R_\ell$ be
(labeled) rectangles where, if $k$ belongs to the $i$th (resp. $j$th)
cycle of $\sigma_1$ ($\sigma_2$) then the height (width) of
$R_k$ is $a^{(1)}_i$ (resp $a^{(2)}_j$). We see that this
surface is determined uniquely 
as a labeled surface. 

We say that a labeled surface is {\em normalized} if the horizontal
and vertical directions are parabolic, and the shortest saddle
connections in the horizontal and vertical direction have equal
length. Associated with such a surface is a pair of Dehn twist 
vectors corresponding to the horizontal and vertical cylinder
decompositions. We have:

\begin{prop}[Veech] \name{prop: parameters determine surface}
The pair of permutations $(\sigma_1, \sigma_2)$ and the Dehn twist
vectors uniquely determine $M$ as a normalized labeled translation surface. 
\end{prop}

\begin{proof}
By the preceding discussion it suffices to show that the permutations
and Dehn twist vectors uniquely determine the vectors
$\vec{a}^{(d)}, \ d=1,2.$ The Dehn twist vectors prescribe the ratios
of inverse moduli of cylinders in a labeled surface; we will show that
determining these ratios uniquely determines $\vec{a}^{(d)}$ up to two
free variables. The assumptions that $M$ is normalized and has area 1
will be used to get a unique solution. 

Let $B = (b_{ij})$ be an integer matrix where $i$ ranges over the cycles
of $\sigma_1$ and $j$ over the cycles of $\sigma_2$, and $b_{ij}$ is
the number of elements in the intersection of the corresponding cycles.
Define 
$$\vec{c}^{(1)} = B \vec{a}^{(2)} \ \ \ \mathrm{and \ }  \ \vec{c}^{(2)} =B^t
\vec{a}^{(1)}.$$
Note that on a labeled surface, the circumference of the $j$th horizontal
cylinder is the sum of 
the widths of the rectangles contained in it. This is just $\sum
b_{ij}a^{(2)}_j$. Thus, if $\vec{a}^{(1)}$ contains the heights of horizontal cylinders
on a labeled surface, then $\vec{c}^{(2)}$ contains the circumferences of
vertical cylinders and vice versa. 

Now let $A_d = 
\diag(\vec{n}^{(d)})$, i.e. an integer diagonal matrix, indexed by
cycles of $\sigma_d$, containing the Dehn twist data. 
In order for the surface determined by $\vec{a}^{(1)}, \vec{a}^{(2)}$
to be parabolic in the horizontal and vertical directions, with the
prescribed Dehn twist vectors, the inverse
modulus $\mu^{(d)}_i$ of the $i$th cylinder must satisfy 
$n^{(d)}_i\mu^{(d)}_i=\mu^{(d)}$. Here 
$\mu^{(1)}$ and $\mu^{(2)}$ are variables.
The relationship between inverse moduli and Dehn twist
numbers gives:
$$\mu^{(d)} \vec{a}^{(d)}=  A_d \vec{c}^{(d)}.$$
Putting these matrix equations together gives 
$$\mu^{(1)} \mu^{(2)} \vec{a}^{(1)} = A_1B A_2 B^t \vec{a}^{(1)},
$$
so $\vec{a}^{(1)}$ is an eigenvector of $E_1= A_1BA_2B^t$ corresponding
to the eigenvalue $\mu^{(1)} \mu^{(2)}$. The
assumption that $\langle \sigma_1, \sigma_2 \rangle$ acts
transitively on $\{1,\ldots,\ell\}$ implies that $E_1$ is 
an irreducible non-negative matrix. By the uniqueness of a positive
eigenvector for an irreducible non-negative matrix \cite{Gantmacher},
$\vec{a}^{(1)}$ is determined, up to a positive scalar, by the matrix
$E_1$. Similarly $\vec{a}^{(2)}$ is uniquely 
determined up to scaling by $E_2 = A_2B^tA_1B$.  

 The shortest horizontal
(vertical) 
saddle connection is a sum of some of the entries in $\vec{a}^{(2)}$
(resp. $\vec{a}^{(1)}$). Requiring these sums to be equal implies that
the scaling of $\vec{a}^{(1)}$ determines the scaling of
$\vec{a}^{(2)}$. Rescaling both vectors by the same factor
$\gamma$ changes the area of the surface
by $\gamma^2$. There is a unique positive value of $\gamma$ which produces a
surface of area 1. 
\end{proof} 


\begin{cor}\name{cor: viii}
If $M \in \NSVT(\beta)$ then condition (viii) of Theorem \ref{thm:
characterization} holds. 

\end{cor}

\begin{proof}
Consider two saddle connections $\xi$ and $\eta$ on $M$ with $|\hol(\xi)
\wedge \hol(\eta)| = \beta_0 <T$. By Propositions \ref{prop: Vorobets
parabolic} the directions of $\xi$ and $\eta$ are parabolic directions
on $M$, by
Proposition \ref{prop:diagonal count} the Dehn twist vectors are
bounded, and by Proposition \ref{prop: rectangle number bound} the
number of rectangles in the corresponding pair of cylinder
decompositions is bounded, hence so are the possibilities for the
corresponding permutations. For any two such pairs $(\xi_i, \eta_i), \
i=1,2$, there is $g \in G$ such that $g\xi_1= \xi_2$ and $g\eta_1 =
\eta_2$. If the corresponding Dehn twist vectors and permutations
are the same, then $gM$ is affinely equivalent to $M$ by Proposition
\ref{prop: parameters determine surface}, i.e. $g \in \Gamma_M$. Thus
the numbers of $\Gamma_M$-orbits of such pairs $(\xi, \eta)$ is finite. 
\end{proof}


\section{Counting cylinder intersection patterns}
In this section we will obtain a combinatorial formula which will
bound the number of affine equivalence classes of
translation surfaces in
$\NSVT(\beta)$ and the sums of the co-areas of the Veech groups of
surfaces in $\NSVT(\beta)$. The first bound will imply Theorem
\ref{thm: NST finite}
and the second will imply Theorem \ref{thm: characterization}, since
the case of half-translation surfaces follows via the orientation
double cover. 

Let $X_\ell$ be the set of pairs of
permutations $(\sigma_1,\sigma_2)$ in $S_\ell$ for which the group
$\langle \sigma_1,\sigma_2\rangle$ acts transitively on the set
$\{1,\ldots,\ell\}$. As we have seen such a pair together with the
Dehn twist vectors determine a normalized labeled surface 
decomposed into $\ell$ rectangles. Since the
horizontal and vertical cylinders of the surface correspond to cycles of
$\sigma_1$ and 
$\sigma_2$, the Dehn twist vector $\vec{n}_j$ may be viewed as a function
from the cycles of $\sigma_j$ to the natural numbers. Let $\mathcal
Y(\ell,\beta)$ denote the set of quadruples
$(\sigma_1,\sigma_2,\vec{n}_1,\vec{n}_2)$ where $(\sigma_1,\sigma_2)\in
X_\ell$ and $\vec{n}_j$ are Dehn twist vectors satisfying the bound in
\equ{eq: twist bounds beta}, in particular, compatible with $M$ being
in $\NSVT(\beta)$.

A permutation $\lambda \in S_\ell$ acts on the set of labels $\{1,
\ldots, \ell\}$, and this induces an action on $(\sigma_1, \sigma_2,
\vec{n}_1, \vec{n}_2)$ by simultaneously conjugating
$\sigma_j$ 
and changing the domains of $\vec{n}_j$. If $\lambda$ fixes the 
quadruple $(\sigma_1, \sigma_2,
\vec{n}_1, \vec{n}_2)$ then re-arranging the rectangles of the
corresponding cylinder decompositions gives rise to a translation
equivalence of $M$. Moreover the relabeling operation gives the entire
translation equivalence class, since any translation equivalence of
a surface preserves the horizontal and vertical
directions, thus preserves the cylinder decompositions and takes
rectangles to rectangles. 

This gives:
\begin{prop} 
\name{prop: order bound}
The cardinality of any $S_\ell$-orbit in $\mathcal{Y}(\ell, \beta)$ 
is at least $(\ell -1)!$ 
\end{prop}

\begin{proof}
For $y = (\sigma_1, \sigma_2, \vec{n}_1, \vec{n}_2) \in
\mathcal{Y}(\ell, \beta),$ let $M$ be the corresponding normalized labeled surface
as in Proposition
\ref{prop: parameters determine surface}. By the above discussion, the
stabilizer in $S_\ell$ of $y$ is isomorphic to the
kernel $K_M$ of the map $D: \Aff(M) \to G$. Since any translation
equivalence permutes the rectangles on $M$, and any equivalence fixing
a rectangle must fix all rectangles, the order of $K_M$ is at most
$\ell$. 
\end{proof} 

\begin{prop}\name{prop: explicit}
Let $\til \NSVT(\beta)$ be the set of affine equivalence classes in
$\NSVT(\beta)$. Then 
$$\left|\til \NSVT(\beta) \right| \le \sum_{\ell
\leq 1/\beta} \frac{1}{(\ell-1)!} \sum_{(\sigma_1,\sigma_2)\in X_\ell} 
\eta(|\sigma_1|,\beta)\eta(|\sigma_2|,\beta),$$
where 
$$\eta(r, \beta)=(8\beta)^{2(1-r)}(1-2\log 
8\beta)^{r-1}$$ 
and 
  $|\sigma|$ is the number of cycles for
$\sigma$.
\end{prop}

\begin{proof} 
Every class in $\til \NSVT(\beta)$ contains a
standard form surface $M$, with a decomposition into $\ell \leq
1/\beta$ rectangles, by Corollary \ref{cor: stand form equiv class}. 
The corresponding combinatorial data gives rise to a quadruple in
$\mathcal{Y}(\ell, \beta)$, and distinct classes in $\til
\NSVT(\beta)$ will give rise to quadruples in distinct
$S_{\ell}$-orbits. 
According
to Proposition \ref{prop:diagonal count}, for each $(\sigma_1,
\sigma_2) \in X_\ell$, 
the number of Dehn twist vectors $\vec{n}_j$ for which $(\sigma_1,
\sigma_2, \vec{n}_1, \vec{n}_2) \in \mathcal{Y}(\ell, \beta)$ is at most 
$\eta(|\sigma_j|, \beta).$
Thus the cardinality of $\mathcal Y(\ell,\beta)$
is bounded by
$$\Phi(\ell, \beta) = \sum_{(\sigma_1,\sigma_2)\in X_\ell}
\eta(|\sigma_1|,\beta)\eta(|\sigma_2|,\beta).$$ 
The size of an $S_\ell$-orbit in $y$ is at least $(\ell - 1)!$ by
Proposition \ref{prop: order bound}. So the number of orbits in
$\mathcal{Y}(\ell, \beta)$ is at most 
$\Phi(\ell, \beta)/(\ell -1)!$ 
The claim follows by summing over $\ell$.
\end{proof}

This explicit estimate, along with a similar bound for
half-translation surfaces, proves Theorem \ref{thm: NST finite}.
We define ${\rm area}\NSVT(\beta)$ to be the sum over the affine
equivalence classes in $\til \NSVT(\beta)$ of the areas of the
Teichm\"uller curves  $\HHH/\Gamma_M$. 

\begin{prop}\name{prop: coarea}
We have 
$${\rm area}\NSVT(\beta)\le 2\pi \sum_{\ell \leq 2/\beta^2}
{\Phi(\ell,\beta)\over (\ell-1)!}.$$

\end{prop}

\begin{proof} Every $M \in \NSVT(\beta)$ which is in standard form
contains at most $2/\beta^2$ rectangles by Corollary 
\ref{cor: stand form area est}. Thus $\sum_{\ell \leq 2/\beta^2} \Phi(\ell,\beta) /
(\ell-1)!$ is an upper bound for the number of standard form
surfaces which are affinely isomorphic to some $M \in
\NSVT(\beta)$. Now the bound follows from Corollary \ref{cor:
standard}. 
\end{proof}

Since the sum of the areas is finite each individual area is finite
and we get a proof of Theorem \ref{thm: characterization}. 

\begin{remark}
Our formulae do not distinguish surfaces on different strata, nor do
they distinguish connected surfaces from disconnected ones. 

\end{remark}


\section{Proof of Theorem \ref{thm: TFAE}}

The following characterization of lattice surfaces is part of
Theorem \ref{thm: TFAE}. 

\begin{thm}
\name{prop: closed implies Veech}
Let $\HH$ be the stratum containing the flat surface
$M$. Then $M$ is a lattice surface if and only if the $G$-orbit of $M$
is closed in $\HH$. 
\end{thm}

\begin{proof} For two proofs of this result, see the sketch in Veech
\cite{Veech - closed} or the proof in \cite[\S 5]{toronto}. 
\end{proof}

\begin{proof}[Proof of Theorem \ref{thm: TFAE}]
 Vorobets \cite{Vorobets} proved the implications
(i) $\Longleftrightarrow$ (v) $\Longrightarrow$ (iv) 
$\Longrightarrow$ (vi) $\Longleftrightarrow$ (ii)
$\Longleftrightarrow$ (iii). Theorem \ref{thm: characterization} shows
that (i) and (vi) are equivalent, Theorem \ref{thm: nst vs nsvt} shows
that (vi) and (ix) are equivalent, and Theorem \ref{prop: closed implies
Veech} gives the equivalence of (i) and (x). Clearly 
that (viii) $\Longleftrightarrow$ (vii) $\Longleftrightarrow$ (ix),
and 
Corollary \ref{cor: viii} implies 
(ix) $\Longleftarrow$ (viii). 
Putting all these together one sees that 
(i)--(x) are equivalent. To conclude the proof we will prove that 
 (i) $\Longrightarrow$ (xi)
$\Longrightarrow$ (vii). 

Assume (i). 
It is well-known that if $\Gamma \subset G$ is a lattice then there is
a compact $K_1 \subset G/\Gamma$ such that 
any geodesic orbit intersects $K$; that is, denoting by $\pi: G \to
G/\Gamma$ the projection map, for every $g \in G$ one has
$$K_1
\cap \{g_t\pi(g) : t \in \R \}\neq \varnothing.$$
Now taking $\Gamma = \Gamma_M$, and letting $\varphi: G/\Gamma \to
\HH$ be the orbit map $\varphi(\pi(g)) = gM,$ we see that
(xi) holds with $K = \varphi(K_1).$

Now assume (xi). 
Let $\eta>0$ be small enough so that for any $M_0 \in K$, the length
of any saddle connection for $M_0$ is at least $\eta.$ 
Let $v_1, v_2 \in \mathrm{hol}(M)$ such that $v_1\wedge v_2 \neq 0,$
and suppose $v_i = \mathrm{hol}(\delta_i)$ where $\delta_i \in \LL.$ 
For $h \in G$, let $l_i(h)$ be the length of $\delta_i$ with respect
to the Euclidean metric on $hM$. 

Let $g \in G$ be
a linear map such that $gv_1$ is horizontal, $gv_2$ is vertical and
both have the same length, which we denote by $c$. By assumption,
there is $t=t(g) \in \R$ with $g_tgM \in K$.  
If $t\geq 0$ then 
$$\eta \leq l_2(g_tg) = e^{-t/2} l_2(g) = e^{-t/2}c,$$
so $c \geq \eta.$ If $t<0$ we apply the same argument with $l_1$
instead of $l_2$ to see that $c \geq \eta.$ 
Since $g$ preserves the two dimensional volume element, 
$$|v_1 \wedge
v_2| = |gv_1 \wedge gv_2| = c^2 \geq \eta^2,$$
which is a positive constant independent of $v_1, v_2$. 
\end{proof}

\begin{remark}
Let $C$ be the number of cusps in
$\HHH/\Gamma_M$. Then the arguments of \cite{Vorobets} show that the 
number of $\Aff(M)$-orbits of triangles as in alternative (iv) of Theorem
\ref{thm: TFAE} is between $C$ and $C (2g+ \left|\Sigma \right| -2)$. 
\end{remark}

\section{The simplest lattice surfaces}
\name{section:
simplest}
In this section we will list the first few lattice surfaces, ordered
as in \equ{eq: alphaM}. We will also list surfaces with removable
singularities. 

Recall that any triangulation of $M$ has $\tau(M)/\pi$ triangles,
implying
\eq{eq: easy upper bound}{\alpha(M) \leq \frac{\pi}{\tau(M)}. }

\begin{prop}
\name{prop: bottom of spectrum}
If
equality holds in \equ{eq: easy upper bound} then $M$ is an arithmetic
lattice surface, all 
triangles on $M$ have the same area, all cylinders in a cylinder
decomposition of $M$ have the same height and all parallel 
saddle connections on $M$ have the same length.
\end{prop}
\begin{proof}
Since 
$\alpha(M)$ is both a lower bound and the 
average of the triangle areas in any triangulation, the area of
each triangle for $M$ must be equal 
to $\alpha(M).$ We will now show that for any $\theta$, the length
$|\sigma|$ of any
$\sigma \in \LL_{M}(\theta)$ is the same. By Theorem \ref{thm: TFAE}, $\theta$
is a parabolic direction 
for $M$, so $\sigma$ is contained in the boundary 
of a cylinder $C$, and there is a singularity $x$ in the
boundary component of $C$ opposite to $\sigma$. If $\sigma' \in
\LL_M(\theta)$ is in the same boundary component of $C$ as $\sigma$,
consider the triangles $\Delta$ and $\Delta'$ with apex $x$ and base $\sigma$ and
$\sigma'$ respectively. Since these triangles have the same height and
area we find that $|\sigma| = |\sigma'|$. Similarly, if $\sigma''$ is a segment on the boundary
component of $C$ opposite to $\sigma$ then a triangle $\Delta''$ with base
$\sigma''$ and apex an endpoint of $\sigma$ has the same area and
height as $\Delta$ so we find that $|\sigma| = |\sigma''|.$ 

Now consider another cylinder $C'$ whose boundary contains
$\sigma$. Since it contains a triangle with base $\sigma$, its height
is the same as that of $C$, and any saddle connection on either of its
boundary components has length $|\sigma|$. Continuing in this fashion and using the
connectedness of $M$ we find that all segments in $\LL_M(\theta)$ have
the same length and all cylinders in the corresponding
cylinder decomposition have the same height.

Now consider $C, \sigma$ as before and let $\sigma'$ be a saddle
connection passing from one 
boundary component of $C$ to another. Write $v = \hol(\sigma), v' =
\hol(\sigma')$, and consider any saddle connection $\lambda$ on
$M$. We can write $\hol(\lambda) = nv + n'v'$, where $n$ (respectively
$n'$) is the number
of times $\lambda$ passes through a cylinder in the direction of
$\sigma'$ (resp. $\sigma$). In particular $\hol(M) \subset \Z v \oplus
\Z v'.$ It follows by \cite[\S 5]{GJ} 
that $M$ is arithmetic. 
\end{proof}

We now list some examples of arithmetic lattice surfaces $M$ and calculate
$\alpha(M)$. Let $\gamma = \tau/\pi$ denote the number of triangles in
a triangulation of $M$, so that by \equ{eq: Gauss Bonnet} we have 
\eq{eq: GB}{
\gamma= 2 \left(\sum r_{\sigma} + \left|\Sigma \right| \right) =
2\left(2g-2+\left|\Sigma \right| \right).
}

\begin{enumerate}
\item
Let $M_1$ be the standard flat torus $[0,1]^2$ with opposite sides
identified, with one marked point at the origin. By \equ{eq: GB} $
\gamma  =2$
so $\alpha(M_1) \leq 1/2.$ On the other hand $\hol(M_1) \subset \Z^2$
so that $|v_1 \wedge v_2| \geq 1$ for any two linearly independent
$v_1, v_2 \in \hol(M_1)$. Since the area of a triangle with sides
$v_1, v_2$ is $|v_1 \wedge v_2|/2$ we see that $\alpha(M_1)=1/2.$
\item
Let $M_2$ be the standard pillowcase with 4 singularities of total
angle $\pi$. Then $\gamma = 4$ so that 
$\alpha(M) \leq 1/4.$ On the other hand $\hol(M_3) \subset
\left(\frac1{\sqrt{2}} \Z \right)^2,$ so that $\alpha(M_3) =1/4.$ 
\item
Let $M_3$ be the standard torus with one marked
point at the origin
and another at $\left(\frac12,0\right)$. Then
$\gamma
= 4$ so that $\alpha(M_2)\leq 1/4$. On the other hand $\hol(M_2)
\subset \Z\left[\frac12\right] \oplus \Z$ so that $\alpha(M_2) = 1/4.$

\ignore{
\item
Let $M_4$ be the standard torus with one marked point at the origin,
one marked point at $\left(\frac13, 0\right)$ and a third marked point at
$\left(\frac23, 0\right)$. Then $\gamma =3$ so that $\alpha(M) \leq
1/6$ and $\hol(M_4) \subset \Z\left[\frac13 \right] \oplus \Z$ so that
$\alpha(M_4) = 1/6.$ 
\item
Let $M_5$ be the half-translation surface obtained by gluing 3 squares
to make a flat torus attached to a pillowcase along a slit (see Figure 2). Then there are two
singularities with total angle $\pi$ and one with total angle $4 \pi$
so that \equ{eq: GB} gives $\gamma =3.$ Also $\hol(M_5) \subset
\Z\left[\frac1{\sqrt{3}} \right] \oplus \Z\left[\frac1{\sqrt{3}}
\right]$ so that 
$\alpha(M_5) = 1/6.$

\item
Let $M_6$ be the translation surface obtained by gluing opposite sides
of an $L$-shaped
polygon made of three equal squares of sidelength
$\frac1{\sqrt{3}}$ (see Figure 2). Then there is
one singularity of cone angle $6 \pi$ and \equ{eq: GB} gives $\gamma
=3,$ so that $\alpha(M_6) \leq 1/6.$ Also $\hol(M_6) \subset
\Z\left[\frac1{\sqrt{3}} \right] \oplus \Z\left[\frac1{\sqrt{3}}
\right] $ so that 
$\alpha(M_5) = 1/6.$
}
\end{enumerate}

We now show:
\begin{prop}
$\NST\left(\frac14 \right)$ consists of the affine equivalence classes
of $M_1, M_2, M_3$. 

\end{prop}

\begin{proof}
The above discussion shows $M_1, M_2, M_3 \in \NST\left(\frac14
\right)$ and it remains to show that if $M$ is a flat surface with
$\alpha = \alpha(M) \geq 1/4$ then $M$ is affinely equivalent to one of the
$M_i$. Let $\gamma$ be as above, so by \equ{eq: easy upper bound} we
have $\gamma \in \{2,4\}.$ If $\gamma =2$ then either $g=1$ and $\left|
\Sigma\right| =1$ or $g=0$ and $\left|\Sigma\right| =3.$ In the first case, since
the moduli space of tori with one marked point is a 
single $G$-orbit, we find that $M$ is affinely equivalent to
$M_1$. The second case does not occur as there is no solution to
\equ{eq: Gauss Bonnet} with three singularities. 

Now suppose $\gamma =4.$ The only solutions to \equ{eq: easy upper
bound} are 
$\left(g=0, \left|\Sigma\right| =4\right)$ and $\left(g=1, \left|
\Sigma \right|=2\right)$. In the
first case it follows from \equ{eq: Gauss Bonnet} that the four
singularities have $r_{\sigma} = -1/2$ so $M$ is a pillowcase. Since
the moduli space of the pillowcase consists of a single $G$-orbit, we
have that $M$ is affinely equivalent to $M_2$. In the second case by
\equ{eq: GB} the two singularities $\sigma_1, \sigma_2$ satisfy either
\begin{itemize}
\item[(i)]
$r_{\sigma_1} = -1/2, r_{\sigma_2}
= 1/2.$
\item[(ii)]
$r_{\sigma_1} = r_{\sigma_2}=0$;
\end{itemize}
In case (i) we obtain a half-translation structure on a torus, but
such a flat surface does not exist (see \cite{MS2}).  
In case (ii) $M$ is a torus with two marked points. Applying an element
of $G$ we may identify $M$ with the unit square, and there is no
loss of generality in assuming that one of the marked points is at the
origin. By Proposition \ref{prop: bottom of spectrum}, if $\sigma$ is
the saddle connection connecting the two marked points inside the unit
square, then there is a parallel segment from the second marked point
to a singularity, of the same length. This implies that the second
marked point is at either of the points $\left(\frac12, 0 \right),
\left(0, \frac12, \right), \left(\frac12, \frac12\right).$ All of
these cases are affinely equivalent to $M_3$. 
\end{proof}

If $\alpha$ is not too small one can continue applying such arguments
to identify $\NST\left(\alpha 
\right)$. For example, in addition to $M_1, M_2,
M_3$, $\NST\left(\frac16\right)$ consists of a torus with two marked
point, a torus with three marked points, a genus 1 half-translation
surface made by gluing a torus and a pillowcase along a slit, and a
genus 2 surface (see Figure 1). All these examples are arithmetic.

\begin{figure}[htp] 
\input{m5m6.pstex_t}
\caption{Two surfaces in $\NST\left( \frac16 \right)$}
\end{figure}


\begin{thebibliography}{99}

\bibitem[Bo]{Bowman} J. Bowman, {\em Doctoral thesis, Cornell
University}, in preparation. 

\bibitem[Cal]{Calta} K. Calta, {\em Veech surfaces and complete
periodicity in genus 2}, J. Amer. Math. Soc. {\bf 17} (2004) 871--908.





\bibitem[EsOk]{EO1} A. Eskin and A. Okounkov, {\em Asymptotics of
numbers of branched coverings of a torus and volumes of moduli spaces
of holomorphic differentials,} Inv. Math.
{\bf 145} (2001), no. 1, 59--103.


 
\bibitem[Ga]{Gantmacher} F. R. Gantmacher, {\bf Theory of matrices,
Vol. 2,} Chelsea publishing, 1959.   




\bibitem[Gu]{Gutkin} E. Gutkin, {\em Billiards on almost integrable
polyhedral surfaces,} Erg. Th. Dyn. Sys. {\bf 4} (1984) 569--584.


\bibitem[GuJu]{GJ} E. Gutkin and C. Judge, {\em Affine maps
of translation surfaces: geometry and arithmetic}, Duke Math. J. {\bf
103}, no.2 (2000) 191--213.







\bibitem[KeSm]{KS} R. Kenyon and J. Smillie, {\em Billiards on
rational-angled triangles}, Comm. Math. Helv. {\bf 75} (2000) 65--108.

\bibitem[KeMaSm]{KMS} S. Kerckhoff, H. Masur and J. Smillie, {\em
Ergodicity of billiard flows and quadratic differentials},
Ann. Math. {\bf 124} (1986) 293--311.







\bibitem[MaSm]{MS2} H. Masur and J. Smillie, {\em Quadratic
differentials with prescribed singularities and pseudo-Anosov
diffeomorphisms}, Comment. Math. Helvitici {\bf 68} (1993) 289--307.

\bibitem[MaTa]{MT} H. Masur and S. Tabachnikov, {\em Rational
billiards and flat structures}, in {\bf Handbook of dynamical systems,
Enc. Math. Sci. Ser.} (2001).

\bibitem[McM]{McMullen} C. McMullen, {\em Billiards and Teichm\"uller
curves on Hilbert modular surfaces} J. Amer. Math. Soc. {\bf 16}
(2003) 857--885.



\bibitem[Na]{Naveh} Y. Naveh, {\em Tight upper bounds on the number of
invariant components}, Isr. J. Math. (to
appear).



\bibitem[Pu]{Puchta} J.-C. Puchta, {\em On triangular billiards,} 
Comment. Math. Helv. {\bf 76} (2001), no. 3, 501--505.



\bibitem[SmWe1]{SW -- Veech} J. Smillie and B. Weiss, {\em Veech's
dichotomy and the lattice property},
preprint (2005).

\bibitem[SmWe2]{toronto} J. Smillie and B. Weiss, {\em Finiteness
results for flat surfaces: a survey and problem list},  (2006) to
appear in {\bf  Partially hyperbolic dynamics, laminations, and 
     Teichm\"uller flow} (Proceedings of a conference, Fields
Institute, Toronto  Jan 2006), G. Forni (ed.)   
 
 


\bibitem[Th]{Thurston} W. Thurston, {\em On the geometry and dynamics
of diffeomorphisms of surfaces}, Bull. AMS (new series) {\bf 19} no. 2
(1988) 417--431.





\bibitem[Ve1]{Veech - alternative} W. A. Veech, {\em
Teichm\"uller curves in moduli space, Eisenstein series and an
application to triangular billiards}, 
Invent. Math. {\bf 97} (1989), no. 3, 553--583.

\bibitem[Ve2]{Veech - closed} W. A. Veech,
{\em Geometric realizations of hyperelliptic curves}, in {\bf
Algorithms, fractals, and dynamics (Okayama/Kyoto, 1992)} , 217--226,
Plenum, New York, 1995.

\bibitem[Ve3]{Veech - spine} W. A. Veech, {\em Bicuspid F-structures
and Hecke groups,}
preprint (2006). 

\bibitem[Vo]{Vorobets} Ya. B. Vorobets, {\em Planar structures and
billiards in rational polygons: the Veech alternative}, (Russian)
Uspekhi Mat. Nauk  {\bf 51}  (1996),  no. 5(311), 3--42;  translation in
Russian Math. Surveys  {\bf 51}  (1996),  no. 5, 779--817.


\bibitem[Zo]{zorich survey} A. Zorich, {\em Flat surfaces}, in {\bf
Frontiers in number theory, physics and geometry,} P. Cartier,
B. Julia, P. Moussa and P. Vanhove (eds), Springer (2006). 

\end{thebibliography}
\end{document}